\documentclass[a4paper,12pt,reqno]{amsart}
 \usepackage[english]{babel}
  \usepackage{amsmath,amsfonts,amssymb,amsthm,bbm}
   \usepackage{graphics,epsfig,psfrag} 
   \usepackage{paralist,subfigure,multirow,float}

	\usepackage{hyperref} 
    \usepackage[active]{srcltx} 
    \usepackage{color,soul} 
    \usepackage[latin1]{inputenc}
    \usepackage[OT1]{fontenc}

\newtheorem{theorem}{Theorem}[section]
\newtheorem{corollary}[theorem]{Corollary}
\newtheorem{lemma}[theorem]{Lemma}
\newtheorem{proposition}[theorem]{Proposition}

\newtheorem*{Landis}{Unique Continuation at Infinity}

\theoremstyle{definition}

\newtheorem{remark}{Remark}

\DeclareMathOperator{\dist}{dist}

\DeclareMathOperator\Tr{Tr}

\DeclareMathOperator*{\esssup}{ess\,sup}
\DeclareMathOperator*{\essinf}{ess{\,}inf}

\def\cprime{$'$}

\def\N{\mathbb{N}}

\def\R{\mathbb{R}}

\let\O=\Omega
\let\e=\varepsilon
\let\vp=\varphi
\let\vt=\vartheta
\let\t=\tilde
\let\ol=\overline

\let\.=\cdot
\let\0=\emptyset

\let\mc=\mathcal

\def\1{\mathbbm{1}}

\newcommand{\su}[2]{\genfrac{}{}{0pt}{}{#1}{#2}}

\def\thm#1{Theorem~\ref{thm:#1}}
\def\seq#1{(#1_n)_{n\in\N}}

\def\inn{\quad\text{in }\;}
\def\for{\quad\text{for }\;}
\def\fall{\quad\text{for all }\;}
\def\step#1{\smallskip{\em Step #1.\,}}

\newenvironment{formula}[1]{\begin{equation}\label{#1}}
                       {\end{equation}\noindent}

\def\Fi#1{\begin{formula}{#1}}
\def\Ff{\end{formula}\noindent}

\def\pe{principal eigenvalue}

\def\MP{maximum principle}
\def\SMP{strong maximum principle}

\def\L{\mathcal{L}}
\def\P{\mc{P}}
\def\l{\lambda_1}

\def\UCI{\textbf{UCI}}

\title[\tiny The Landis conjecture with sharp rate of decay]
{The Landis conjecture \\ with sharp rate of decay}
\author{Luca  Rossi}
\thanks{CNRS, EHESS, PSL Research University, CAMS, Paris, France}

\begin{document}

\begin{abstract}
	The so called Landis conjecture states that if a solution of the equation 
	$$\Delta u+V(x)u=0$$
	in an exterior domain decays faster than $e^{-\kappa|x|}$, 
	for some $\kappa>\sqrt{\sup |V|}$, then it must be identically equal to $0$.
	This property can be viewed as a unique continuation at infinity (UCI) for solutions
	satisfying a suitable exponential decay.
	The Landis conjecture was disproved by Meshkov in the case of complex-valued functions,
	but it remained open in the real case.
	In the 2000s, several papers have addressed the issue of the UCI for linear elliptic operators
	with real coefficients. The results that have been obtained require some
	kind of sign condition, either on the solution or on the zero order coefficient of the equation. 
	The Landis conjecture is still open nowadays in its general form.

	In the present paper, we start with considering a
	general (real) elliptic operator in dimension $1$.
	We derive the UCI property with a rate of decay 
	$\kappa$ which is sharp when the coefficients of the operator are constant.
	In particular, we prove the Landis conjecture in dimension $1$,
	and we can actually reach the threshold value $\kappa=\sqrt{\sup |V|}$.
	Next, we derive the UCI property --~and then the Landis conjecture~--
	for radial operators in arbitrary dimension.
	Finally, with a different approach, we prove the same result 
	for positive supersolutions of general elliptic equations.
\end{abstract}

\maketitle

\vspace{-17pt}


\section{Introduction}

In \cite{KLandis}, Kondrat{\cprime}ev and Landis asked the following question:
if $u$ is a solution of the equation 
\Fi{Delta+V}
\Delta u+V(x)u=0
\Ff
in the exterior of a ball in $\R^N$, is it true that the condition
\Fi{HypL}
\exists\kappa>\sqrt{\sup |V|},\qquad
u(x) \prec e^{-\kappa|x|},
\Ff
necessarily implies $u\equiv0$ ?
Here and in the sequel, the expression $u \prec v$ means $u(x)/v(x)\to 0$ as $|x|\to\infty$.
They also addressed the same question under the stronger requirement 
that $u(x) \prec e^{-\kappa|x|}$ for all $\kappa>0$.

The question is motivated by the trivial observation that
in dimension $N=1$ with~$V$ constant, 
decaying solutions can only exist if $V<0$, and they decay as
$\exp(-\sqrt{|V|}|x|)$.
Hence, in such case, one can even take $\kappa=\sqrt{\sup |V|}$ in condition~\eqref{HypL}.
This is no longer true in higher dimension: 
the bounded, radial solution of 
$\Delta u-u=0$ outside a ball,
which can be expressed in terms of the modified Bessel function of second kind,
decays like $|x|^{-\frac{N-1}2}e^{-|x|}$.
As we will see in the sequel,
this discrepancy between one and multidimensional cases holds true
for general elliptic equations with variable coefficients.

The question by Kondrat{\cprime}ev and Landis received a negative answer in the 
paper~\cite{Mesh} by Meshkov. There, the author exhibits two complex-valued, bounded functions $u,V\not\equiv0$ 
satisfying the equation~\eqref{Delta+V}, with 
$|u(x)|\leq \exp(-h|x|^\frac43)$ for some $h>0$. 
On the other hand, Meshkov shows that the power $4/3$ is optimal,
in the sense that if one strengthens the decay condition
by $u(x)\prec \exp(-|x|^{\frac43+\e})$ for some $\e>0$, then necessarily $u\equiv0$.
These results provide a complete picture in the complex case.

The conjecture has been brought back to attention in the 2000s
by the works of Bourgain and Kenig \cite{BouKen}
and Kenig \cite{KenigLandis}.
In the former, the authors improved Meshkov's positive answer 
in the case of real-valued functions, pushing the decay condition up to
$u(x)\leq\exp(-h|x|^{\frac43}\log(|x|))$. 
However, there is not an analogue
of Meshkov's counterexample (nontrivial solutions with exponential decay with 
power larger than~$1$) in the real case.
This fact led Kenig to ask in \cite[Question 1]{KenigLandis} whether,
{\em in the real case}, the condition 
$$u(x)\prec e^{-|x|^{1+\e}}\fall\e>0$$
necessarily implies $u\equiv0$. 
Observe that this condition is stronger than the original requirement
\eqref{HypL} of \cite{KLandis}. However, even this weaker conjecture
is still open nowadays, except for some particular situations.
Kenig, Silvestre and Wang proved it in~\cite{KSW}
in dimension $N=2$ and under the
additional assumption that $V\leq0$.
The condition on the decay is 
$ u(x) \prec e^{-h|x|(\log|x|)^2}$ for some $h>0$, 
hence the result does not answer the original question in \cite{KLandis}.
In the case of equations set in the whole space $\R^2$,
the authors are able to handle more general uniformly elliptic
operators, still assuming $V\leq0$, see also~\cite{DKW}.
As observed in \cite{ABG}, for equations in the whole space,
this hypothesis implies that $u\equiv0$ just assuming that $u\prec1$,
as an immediate consequence of the maximum principle. 
We point out that the results of~\cite{KSW,DKW} are deduced from
a quantitative estimate which implies that
the set where a nontrivial solution is bounded from below by 
$e^{-h|x|(\log|x|)^2}$ is relatively dense in $\R^2$.

In this paper, we deal with uniformly elliptic operators with real coefficients,
whose general form is
$$\L u=\Tr(A(x)D^2 u)+q(x)\. Du+V(x)u,$$
defined on an {\em exterior domain} $\O\subset\R^N$, 
i.e., a connected
open set with compact complement. 
For general operators of this type, it is known 
since the work of Pli\'s~\cite{Plis} that the question asked by 
Kondrat{\cprime}ev and Landis has a (dramatically) negative answer.
Namely, Pli\'s exibhits an operator $\L$ in $\R^3$ with a H\"older-continuous matrix 
field~$A$ and smooth terms $q$, $V$ which admits a nontrivial solution 
vanishing identically outside a ball.
This is an astonishing counterexample to the property of {\em unique continuation 
at infinity}. Here we consider the following definition of such property.

\begin{Landis}[\UCI]
	We say that a given equation satisfies the {\em \UCI} with a rate of decay $\kappa$, 
	if the unique solution satisfying
	$$ u(x) \prec e^{-\kappa|x|}$$
	is $u\equiv0$.
\end{Landis}
The \UCI\ implies that two distinct solutions cannot have the same
behaviour at infinity up to an additive term decaying sufficiently fast.
Owing to Pli\'s' counterexample, the only hope to 
derive the \UCI\ is by requiring some additional hypotheses on the operator.
For instance, the results by Kenig and collaborators are restricted 
to dimension $N=2$, whereas the counterexample is in dimension $3$.
Another possible way to avoid the counterexample of \cite{Plis} is by assuming
a suitable regularity of the diffusion matrix $A$. It is indeed known that the
pathological situation of \cite{Plis}
cannot arise when $A$ is Lipschitz-continuous, see \cite{GLunique}.
However, these restrictions do not seem to be useful in an approach based on 
the comparison principle and Hopf's lemma, which is the one 
adopted in the present paper.

In the very recent paper \cite{ABG},
Arapostathis, Biswas and Ganguly attack the problem using probabilistic tools.
They derive the \UCI\ under the additional assumption that 
$u\geq0$, or, if $\O=\R^N$, that $\l\geq0$, where~$\l$ is the {\em generalised
	\pe} of the operator $-\L$, see the definition \eqref{l1} below.
We point out that the condition $\l\geq0$ is more general than both $V\leq0$ and $u\geq 0$.
The threshold for the rates of decay $\kappa$ obtained in \cite{ABG}
depends on the coefficients of~$\L$ and it
is not optimal when $q\not\equiv0$, see the discussion in the next subsection.

\subsection{Statement of the main results}
We consider a general (real) elliptic operator
$$\L u=\Tr(A(x)D^2 u)+q(x)\. Du+V(x)u,$$ 
in an exterior domain $\O\subset\R^N$.
We always assume that
the matrix field $A$ is bounded, continuous and uniformly elliptic, i.e., its smallest eigenvalue
$$\alpha(x):=\min_{\xi\in \R^N\setminus\{0\}}\frac{A(x)\xi\.\xi}{|\xi|^2}$$
satisfies $\inf_\O\alpha>0$. The vector field $q$ and the potential $V$ belong to $L^\infty(\O)$.
Solutions, subsolutions and supersolutions
of the equation $\L u=0$ are always assumed to belong to $W_{loc}^{2,N}$
and to satisfy respectively $\L u=0$, $\L u\geq 0$ and $\L u\leq 0$ a.e.
Observe that, by elliptic estimates, solutions 
actually belong to $W_{loc}^{2,p}$ for all $p<+\infty$.
In general, when referred to measurable functions, 
the equalities or inequalities are understood to hold a.e., and $\inf$, $\sup$
stand for $\essinf$, $\esssup$.

Our first result concerns the case of dimension $N=1$, where $\L$ is given by
$$
\L u=\alpha (x)u''+q(x)u'+V(x)u,
$$
defined on the half-line $\R^+=(0,+\infty)$.
\begin{theorem}\label{thm:1}
	In the case $N=1$, any solution of $\L u=0$ in $\R^+$ satisfies
	$$\varlimsup_{x\to+\infty}|u(x)|e^{\kappa x}\geq |u(x_0)|e^{\kappa x_0},$$
	for every $x_0>0$, where 
	\Fi{k1}
		\kappa=\sup\frac{|q|}{2\alpha}+\sqrt{\sup\frac{|q|^2}{4\alpha^2}
		+\sup\frac{|V|}\alpha}.
	\Ff
\end{theorem}
As a consequence, the \UCI\ holds when $N=1$ with the rate of decay~\eqref{k1}. 
Let us make some comments about this rate of decay.
If the coefficients $\alpha,q,V$ are constant with $q\geq0$ and $V\leq0$
then $\kappa$ in \eqref{k1} is precisely the rate of decay of  
solutions at~$+\infty$. Hence, our result is sharp in that case.
The fact that we are able to obtain the equality in \eqref{k1}
instead of the strict inequality `$>$' is actually surprising for us.
As explained before, this threshold rate of decay cannot be obtained in higher dimension. 
We also derive a result in the spirit of the quantitative estimate of 
\cite[Theorem 1.2]{KSW}, which implies that 
the inequality $|u(x)|>e^{-\kappa'x}$, for any $\kappa'$ larger than $\kappa$ in \eqref{k1}, 
holds in a relatively dense set (see Proposition \ref{dense} below).
We recall that the lower bound in \cite{KSW} is $e^{-\kappa|x|\log(|x|)}$.

\thm{1} can be readily extended to {\em radial solutions} (i.e., 
of the form $u(x)=\phi(|x|)$) for general elliptic equations in higher dimension.
\begin{corollary}\label{cor:urad}
	Let $u$ be a nontrivial, radial solution of $\L u=0$ in an exterior domain~$\O$.
	Then, 
	$$\varlimsup_{|x|\to+\infty}|u(x)|e^{\kappa|x|}=+\infty,$$
	for all $\kappa$ satisfying
	$$\kappa>\varlimsup_{|x|\to+\infty}\frac{|q|}{2\alpha}+\sqrt{\varlimsup_{|x|\to+\infty}
		\frac{|q|^2}{4\alpha^2}
		+\varlimsup_{|x|\to+\infty}\frac{|V|}\alpha}.$$
\end{corollary}

Next, we extend the \UCI\ property to {\em radial operators} in arbitrary dimension.
This is achieved by applying our one-dimensional result to the spherical harmonic decomposition of the solution. 
This idea of considering the harmonic decomposition of the solution 
is not new in the context of the unique continuation property, see~\cite{Mu54}.

\begin{theorem}\label{thm:Lrad}
Assume that $\L$ is of the form
$$\L u=\Delta u+q(|x|)\frac x{|x|}\.\nabla u+V(|x|) u.$$
Let $u$ be a nontrivial solution of $\L u=0$ in an exterior domain~$\O$.
Then, 
$$\varlimsup_{|x|\to+\infty}|u(x)|e^{\kappa|x|}=+\infty,$$
for all $\kappa$ satisfying
$$\kappa>\varlimsup_{r\to+\infty}\frac{|q|}{2}+\sqrt{\varlimsup_{r\to+\infty}
	\frac{|q|^2}{4}
	+\varlimsup_{r\to+\infty}|V|}.$$
\end{theorem}
This theorem implies that the Landis conjecture holds for radial potentials $V$.
It~also entails the result in the case of constant coefficients $A,q,V$,
	by a simple change of coordinate system which transforms $A$ into the identity matrix 
	and then multiplying the solution by a suitable exponential in order to absorb 
	the drift term.

We then focus on solutions with a given sign. This makes the problem much simpler,
because the sign condition allows one to directly
use some comparison arguments in order to control the decay of the solution.
One of the consequences of this is that
the result applies to supersolutions.
\begin{theorem}\label{thm:u>0}
	Let $u$ be a positive supersolution of $\L u=0$ in an exterior domain~$\O$.
	Then, 
	$$u(x)\succ e^{-\kappa|x|},$$
	for all $\kappa$ satisfying
	\Fi{k2}
	\kappa>\varlimsup_{|x|\to\infty}
	\left(\frac{|q|}{2\alpha}+\sqrt{\frac{|q|^2}{4\alpha^2}
		+\frac{|V|}\alpha}\right).
	\Ff
\end{theorem}
Actually, the above result is derived in Section \ref{sec:parabolic}
in the more general framework of 
ancient supersolutions of parabolic equations.
We remark that \thm{u>0} provides a lower bound on the decay of $u(x)$
as $|x|\to\infty$, whereas our previous estimates
only hold along some diverging sequences, which is natural because
the functions there are allowed to change sign.
An analogous result to \thm{u>0}
is derived in \cite[Corollary~4.1]{ABG}, using a completely
different method based on the stochastic representation of solutions, 
but with the following threshold for the rate of decay:
$$
\kappa>
\varlimsup_{|x|\to\infty}\frac{|q|}{\alpha}+
\varlimsup_{|x|\to\infty}\sqrt{\frac{|V|}\alpha}.
$$
This threshold is larger than or equal to the one in \eqref{k2}, 
and it is not sharp, even for operators with constant coefficients, if $q\not\equiv0$.
Furthermore, the threshold in \eqref{k2} is
expressed in terms of the $\varlimsup$ of the combination
of $q,V,\alpha$, which is in general smaller than the 
combination of their $\varlimsup$.

Finally, as in \cite{ABG}, we extend the result to sign-changing solutions
under the assumption that the generalised
\pe\ $\l$ is nonnegative. The latter is defined as follows:
\Fi{l1}
\l:=\sup\{\lambda\ :\ \exists\vp>0,\ (\L+\lambda)\vp\leq0\text{
	in }\O\},
\Ff
or it can be equivalently defined as the limit as $r\to+\infty$
of the classical \pe\ in $\O\cap B_r$ under Dirichlet boundary condition
if $\O$ is smooth (see, e.g., \cite{Furusho1,Agmon}, or the more recent paper \cite{BR4}).
Clearly, the hypothesis of \thm{u>0} yields $\l\geq0$,
and so does condition $V\leq0$.
\begin{theorem}\label{thm:l1}
	Let $u$ be a nontrivial supersolution of $\L u=0$ 
	in an exterior domain~$\O$. Assume that $\l\geq0$ and that either $\O=\R^N$ or that
	$$\varliminf_{x\to\partial\O}u(x)\geq0.$$
	Then,
	$$\varlimsup_{|x|\to+\infty}|u(x)|e^{\kappa|x|}=+\infty,$$
	for all $\kappa$ satisfying \eqref{k2}.
\end{theorem}
We point out that \cite{ABG} only covers the case $\O=\R^N$ 
(with a larger threshold for~$\kappa$).
Here, in the case of an exterior domain, we do not assume any regularity of the boundary,
but we need to impose a sign for the solution there.
In order to deal with the lack of regularity of the domain, 
we make use of the maximum principle in small domains
derived by Berestycki, Nirenberg and Varadhan in \cite{BNV}, building on an idea of Bakelman. 
The result of \cite{BNV} actually provides a `refined' maximum principle, in which the boundary 
condition is understood in a suitable weak sense. We believe this should allow one 
to relax the boundary condition in our \thm{l1} too.


The following table summarises all the cases in which we 
derive the \UCI\ property, with the corresponding values of the rate of decay $\kappa$.

\setlength{\tabcolsep}{15pt}

\begin{table}[H]\small
  \caption{{\em Validity of the }\UCI }
\begin{tabular}{ l | l }
    	
	\noalign{\vspace{-7pt}}\hline\noalign{\medskip}
		
	$N=1$ &	
	$\displaystyle
	\kappa=\sup\frac{|q|}{2\alpha}+\sqrt{\sup\frac{|q|^2}{4\alpha^2}
	+\sup\frac{|V|}\alpha}$	
	\\ 
	\noalign{\medskip}\hline\noalign{\medskip}
	
	$u$ is radial,&	 
	\multirow{3}{*}
	{$\displaystyle
	\kappa>\varlimsup_{|x|\to\infty}\frac{|q|}{2\alpha}+\sqrt{\varlimsup_{|x|\to\infty}
	\frac{|q|^2}{4\alpha^2}
	+\varlimsup_{|x|\to\infty}\frac{|V|}\alpha}$}
	\\
	or $\L$ is radial, & \\
	\vspace{2pt}or $\L$ has constant coefficients & \\
	\noalign{\medskip}\hline\noalign{\medskip}
	

	$u\geq0$, & 
	\multirow{3}{*}
	{$\displaystyle
	\kappa>\varlimsup_{|x|\to\infty}
	\left(\frac{|q|}{2\alpha}+\sqrt{\frac{|q|^2}{4\alpha^2}
	+\frac{|V|}\alpha}\right)$}
	\\ 
	or $\O=\R^N$ and $\l\geq0$, & \\
	or $\displaystyle\varliminf_{x\to\partial\O}u(x)\geq0$ and $\l\geq0$ & \\
	\noalign{\medskip}\hline	
\end{tabular}
\label{tab:kappa}
\end{table}




\bigskip
\section{The one-dimensional case}

In this section, $N=1$ and the operator $\L$ is defined in the half-line $\R^+\!=(0,+\infty)$~by 
$$
\L u=\alpha (x)u''+q(x)u'+V(x)u.
$$
We assume that $\alpha,q,V\in L^\infty(\R^+)$ and that $\inf \alpha>0$. 
We let $\beta,\gamma$ denote the following quantities:
$$\beta:=\sup_{\R^+}\frac{|q|}{\alpha},\qquad \gamma:=\sup_{\R^+}\frac{|V|}{\alpha}.$$

The strategy we employ to prove the \UCI\ property relies on the comparison 
with suitable solutions for 
the following nonlinear operators with constant coefficients:
$$
\L_* u:= u''-\beta|u'|-\gamma|u|,
$$
$$
\L^* u:= u''+\beta|u'|+\gamma|u|.
$$
These are the ``extremal'' operators associated with $\L$, 
in the sense that 
$$\L_*\leq\L\leq\L^*,$$ that is,
solutions for $\L$ are supersolutions for $\L_*$ and subsolutions for $\L^*$.
As a matter of fact, we will actually deal with functions $u$ satisfying more generally
$\L_* u\leq0\leq\L^* u$ rather than $\L u=0$.
Concerning the regularity, we have that if $u\in W^{2,1}_{loc}(\R^+)$ solves
$\L u=0$, then $u'\in C(\R^+)$ and therefore, using the equation, we find 
that $u\in W^{2,\infty}_{loc}(\R^+)$. Thus, we work in this regularity framework.

Positive, decreasing solutions of $\L_*=0$ and negative, increasing solutions of $\L^*=0$ satisfy
$u''+\beta u'-\gamma u=0$. They decay at $+\infty$ as 
$e^{-\kappa x}$, with $\kappa$ given by \eqref{k1}, i.e.,
\Fi{k1D}
\kappa=\frac{\beta}2+\sqrt{\frac{\beta^2}4+\gamma}.
\Ff
Our aim is to show that the same $\kappa$ provides a lower bound for the exponential rate of decay
(along some sequence) for sign-changing functions satisfying $\L_* u\leq0\leq\L^* u$.
Throughout this section, $\kappa$ denotes the above quantity.

The comparison principle between sub and supersolutions for the extremal operators 
requires the positivity of the supersolution.
This condition implies that the generalised \pe\ of the nonlinear operator has the sign
that ensures the validity of the \MP.
We further require that the derivatives of the functions do not vanish simultaneously, in order
to reduce to the linear case.
\begin{proposition}\label{pro:MP}
	Let $(a,b)$ be a bounded interval and let $u_1,u_2\in W^{2,\infty}((a,b))$ satisfy
	$$\max_{[a,b]} u_1>0,\qquad\min_{[a,b]} u_2>0,\qquad |u_1'|+|u_2'|\neq0 \inn(a,b),$$
	and either
	$$\L_* u_2\leq0\leq\L_* u_1\qquad\text{or}\qquad
	\L^* u_2\leq0\leq\L^* u_1 \ \inn(a,b).$$
	Then
	$$\max_{[a,b]}\frac{u_1}{u_2}=\max\left\{\frac{u_1(a)}{u_2(a)},\frac{u_1(b)}{u_2(b)}\right\}.$$
	Moreover, unless $u_1/u_2$ is constant, the above maximum cannot be attained at some interior point
	and in addition if it is attained at $y$ (resp.~$z$), there holds
	$$\frac{u_1'(a)}{u_1(a)}<\frac{u_2'(a)}{u_2(a)}\qquad
	\text{\bigg(resp. \ }\frac{u_1'(b)}{u_1(b)}>\frac{u_2'(b)}{u_2(b)}\text{ \bigg)}.$$
\end{proposition}

\begin{proof}
	The argument is classical, see e.g.~\cite[Theorem~2.10]{Max}, even if here we deal with nonlinear operators. 
	We define $w:=u_1/u_2$. Assume that $M:=\max_{[a,b]}w>0$ is attained at some interior point $x_0\in(a,b)$. 
	It follows that $u_1(x_0)>0$. Moreover, $w'(x_0)=0$, that is,
	$u_1'(x_0)u_2(x_0)=u_1(x_0)u_2'(x_0)$, which, because $|u_1'|+|u_2'|\neq0$, implies that 
	$u_1'$ and $u_2'$ have the same strict sign.
	This means that $\L_* u_j=\t \L u_j$ or $\L^* u_j=\t \L u_j$ for $j=1,2$ in some neighborhood $J$ of $x_0$, 
	where $\t\L$ is a linear operator of the type $\t\L u=u''+\t\beta u'+\t\gamma u$. We then compute, in $J$,
	\[\begin{split}
	0\leq\t\L u_1 &=\t\L (u_2 w )\\
	&=u_2\bigg(w''+\Big(2\frac{u_2'}{u_2}+\t\beta\Big)w'\bigg)+
	\big(u_2''+\t\beta u_2'+\t\gamma u_2\big)w\\
	&=u_2\bigg(w''+\Big(\frac{u_2'}{u_2}+\t\beta\Big)w'+\frac{\t\L u_2}{u_2}\,w\bigg).
	\end{split}\]
	This means that $w$ is a subsolution in $J$ of an equation with nonpositive zero order term.
	We can therefore apply the strong maximum principle and infer that 
	$w\equiv M$ in $J$. We have thereby shown that the set where $w$ attains its maximum is both open and closed in $(a,b)$, 
	i.e., it is either empty or it coincides with the whole~$(a,b)$.
	
	It remains to prove the last statement of the proposition.
	Suppose that $w$ is not constant and that its maximum is attained at $a$ 
	(the other case is analogous).
	Then the Hopf lemma (see, e.g.,~\cite{GT}) implies that $w'(a)<0$, that is,
	$$0>\frac{u_1'(a) u_2(a)-u_2'(a)u_1(a)}{u_2^2(a)},$$
	from which the desired inequality follows because $u_1(a),u_2(a)>0$.
\end{proof}

Proposition \ref{pro:MP} will be used to compare sub and supersolutions of Cauchy problems. 
Let us anticipate how this will be done, since, contrary to the 
usual application, we will use subsolutions to get upper bounds and supersolutions 
to get lower bounds.
Namely, let $u,v$ be respectively 
a subsolution and a positive, monotone supersolution 
of an extremal operator such that $u(a)=v(a)$ and
$u'(a)>v'(a)$. Then $v$ is smaller than $u$ in a right neighbourhood of $y$.
If they cross at some point $b>a$, then we would get a contradiction with the 
\MP\ of Proposition \ref{pro:MP}.
This means that $u>v$ to the right of $y$, as long as $v,v'$ do not vanish.

The idea of the proof of \thm{1} consists in distinguishing the region where~$u$ 
is less steep than the exponential $e^{-\kappa x}$ from the points where it is steeper.
We recall that $\kappa$ is given by \eqref{k1D}.
The steepness refers to the ratio $-u'(x)/u(x)$. On one hand, in the first case $u$ decays at most as 
$e^{-\kappa x}$. On the other, we will show that if $u$ is steeper than $e^{-\kappa x}$
at a point $\bar x$ then $|u|$ hits the $x$-axis at some $\t x>\bar x$ with a certain slope
and then it eventually crosses back the exponential function $|u(\bar x)|e^{-\kappa x}$
at a later point.
This `bouncing property', depicted in Figure \ref{fig:bouncing},
is the object of the next lemma.
\begin{figure}[H]
\begin{center}
\vspace{-10pt}
\includegraphics[height=6.5cm]{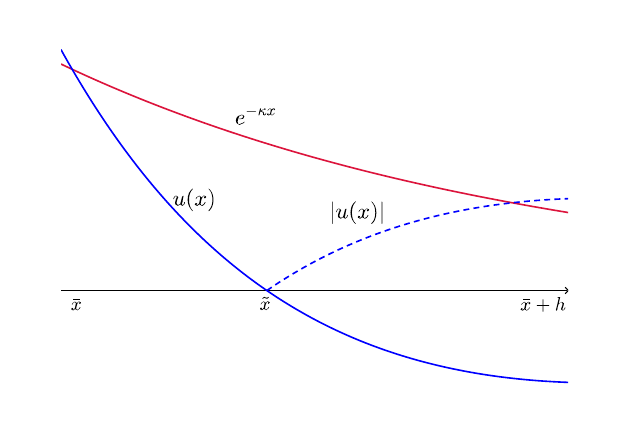}
\vspace{-10pt}
\caption{The `bounce' of $|u|$ when steeper than $e^{-\kappa x}$.}
\label{fig:bouncing}
\end{center}
\end{figure}
\begin{lemma}\label{bouncing}
\vspace{-10pt}
	Let $u\in W^{2,\infty}_{loc}(\R^+)$ satisfy $\L_* u\leq0\leq\L^* u$ in $\R^+$ 
	and assume that there exists $\bar x>0$ for which the 
	following occur:
	$$u(\bar x)\neq0,\qquad
	-\frac{u'(\bar x)}{u(\bar x)}>\kappa.$$
	Then there exists $h>0$ such that
	$$-\frac{u(\bar x+h)}{u(\bar x)}>e^{-\kappa h}.$$
\end{lemma}

\begin{proof}	
	If $\kappa=0$, i.e.~$q\equiv V\equiv0$, the result trivially holds.
	Suppose that $\kappa\neq0$.
	Up to replacing the function $u$ with $u(\bar x+\.)/u(\bar x)$, we can assume 
	without loss of generality that $\bar x=0$ and that $u(0)=1$,
	$u'(0)<-\kappa$. With this change, the proof amounts to showing that $-u(h)>e^{-\kappa h}$
	for some $h>0$.
	
%
	Consider a solution $v$ of the equation 
	$v''+\beta v'-\gamma v=0$, that is,
	$$v(x)=Ae^{-\kappa x}+Be^{\lambda x},$$
	for some $A,B\in\R$ and 
	$$\lambda=-\frac{\beta}2+\sqrt{\frac{\beta^2}4+\gamma}\geq0.$$
	Imposing $v(0)=1$ and $v'(0)=-\kappa'$, where $\kappa'$ is a fixed number satisfying 
	$\kappa<\kappa'<-u'(0)$, reduces to the system
	$$\begin{cases}
		A+B=1\\
		-A\kappa+B\lambda=-\kappa'.
	\end{cases}$$
	It follows that $B(\kappa+\lambda)=\kappa-\kappa'<0$, whence $B<0$ and therefore $A>1$.
	As a consequence, $v'$ is negative and $v$ vanishes at some point $\xi>0$, which means that
	$\L_* v=0$ in~$(0,\xi)$. 
	Applying Proposition \ref{pro:MP} with $u_1=v$, $u_2=u$ we deduce that,
	for every $b>0$ such that $u>0$ in $[0,b]$, there holds
	$$\max_{[0,b]}\frac{v}{u}=\max\left\{1,\frac{v(b)}{u(b)}\right\}.$$
	The above left-hand side is larger than $1$ because $v'(0)>u'(0)$. It follows that 
	$v(b)>u(b)$. This means that the inequality $u<v$ holds as long as $u$ remains
	positive, and therefore $u$ must vanish somewhere in $(0,\xi]$. 
	Let $\t x$ denote the first zero of $u$ in $(0,\xi]$.
	
	We claim that $u'<v'$ in $(0,\t x)$. 	
	Assume that this is not the case. Let $\zeta$ be the smallest point in $(0,\t x)$
	where $u'=v'$. In the interval $(0,\zeta)$ we have that $u'<v'<0$ and $0<u<v$,
	from which we obtain
	$$(u-v)''\leq-\beta(u-v)'+\gamma(u-v)\leq -\beta(u-v)'.$$
	This inequality can be rewritten as $(\log(v'-u'))'\geq-\beta$. Therefore, $ \log(v'-u')$
	is bounded from below in $[0,\zeta)$, contradicting $u'(\zeta)=v'(\zeta)$.
	
	Let us call $\t u:=-u$.
	The following properties hold at the point $\t x$:
	$$\t u(\t x)=0,\qquad
	\t u'(\t x)\geq -v'(\t x)= A\kappa e^{-\kappa\t x}-B\lambda e^{\lambda\t x}>
	\kappa e^{-\kappa\t x}.$$
	We will derive the desired lower bound for $\t u$ by comparison
	with a solution associated with the extremal operator~$\L^*$,
	for which the following holds.
	\begin{lemma}\label{lem:L*}
		Assume that $\kappa>0$. Let $w$ be the solution of the Cauchy problem
		\Fi{CP}
		\begin{cases}
			\L^* w=0 & \text{in }\;\R^+\\
			w(0)=0\\
			w'(0)=1.
		\end{cases}
		\Ff
		Then there exists $\hat x>0$ such that $w'>0$ in $(0,\hat x)$ and moreover
		$$w(\hat x)>\frac1\kappa e^{-\kappa\hat x}.$$
	\end{lemma}
	Let us postpone the proof of Lemma \ref{lem:L*} until the end of the current one. 
	Consider the function $w$ and the number $\hat x$ provided by the lemma. 
	Then, for $\e\in(0,\hat x)$,~set
	$$u_2(x):= \kappa e^{-\kappa\t x}w(x-\t x+\e).$$
	The functions $\t u,u_2$ satisfy $\L^*u_2=0\leq \L^*\t u$ in $[\t x,\t x+\hat x-\e]$.
	We can therefore apply Proposition~\ref{pro:MP} and derive
	\Fi{MPshifted}
	\max_{[\t x,\t x+\hat x-\e]}\frac{\t u}{u_2}=
	\max\left\{\frac{\t u(\t x)}{u_2(\t x)},\frac{\t u(\t x+\hat x-\e)}{u_2(\t x+\hat x-\e)}\right\}=
	\max\left\{0,\frac{\t u(\t x+\hat x-\e)}{u_2(\t x+\hat x-\e)}\right\}.
	\Ff
	In order to get a lower bound for the above left-hand side, we compute
	\[\begin{split}
	\t u(\t x+\sqrt\e)-u_2(\t x+\sqrt\e) &= \t u'(\t x)\sqrt\e-u_2'(\t x-\e)(\sqrt\e+\e)+o(\sqrt\e)\\
	&=\sqrt\e\big(\t u'(\t x)-\kappa e^{-\kappa\t x}\big)+o(\sqrt\e).
	\end{split}\]
	Recalling that $\t u'(\t x)>\kappa e^{-\kappa\t x}$, 
	we deduce that $\t u(\t x+\sqrt\e)>u_2(\t x+\sqrt\e)$
	for $\e$ small enough. It then follows from \eqref{MPshifted} that 
	$\t u(\t x+\hat x-\e)>u_2(\t x+\hat x-\e)$ for $\e$ sufficiently small, that is,
	$-u(\t x+\hat x-\e)>\kappa e^{-\kappa\t x} w(\hat x)$.
	Letting $\e\to0$, we finally get the desired inequality
	$$-u(\t x+\hat x)\geq \kappa e^{-\kappa\t x} w(\hat x)>e^{-\kappa(\t x+\hat x)}.$$
\end{proof}

\begin{proof}[Proof of Lemma \ref{lem:L*}]
	The function $w$ is positive and increasing up to a value $\hat x\in(0,+\infty]$.
	In the interval $(0,\hat x)$, $w$ satisfies 
	$w''+\beta w'+\gamma w=0$. 
	We treat the different types of solutions of this equation separately.
	
		\medskip
	{\em Case $\beta^2>4\gamma$.}
	
	\noindent 
	In this case the solution $w$ is given in $(0,\hat x)$ by
	$$w(x)=\frac1{2\omega}\, e^{-\frac{\beta}{2}x}\left(e^{\omega x}-e^{-\omega x}\right),
	\quad\text{with }\ \omega:=\frac{\sqrt{\beta^2-4\gamma}}{2}>0.$$
	If $\gamma=0$ then $\omega=\frac\beta2$ and therefore $w$ is increasing,
	which immediately entails the conclusion of the lemma.
	If $\gamma>0$ then $\omega<\frac\beta2$ and we find that
	$\hat x$ is a critical point for $w$, characterised by
	$$\frac{\beta}{2}\left(e^{\omega \hat x}-e^{-\omega \hat x}\right)=
	\omega\left(e^{\omega \hat x}+e^{-\omega \hat x}\right).$$
	Using this equivalence, we derive
	\[\begin{split}
	w(\hat x)e^{\kappa\hat x} &=
	\frac1{2\omega}\, e^{\left(\kappa-\frac{\beta}{2}\right)\hat x}
	\left(e^{\omega \hat x}-e^{-\omega \hat x}\right)\\
	&=\frac1\beta e^{\left(\kappa-\frac{\beta}{2}\right)\hat x}
	\left(e^{\omega \hat x}+e^{-\omega \hat x}\right)	\\
	&>\frac{2}\beta e^{\left(\kappa-\frac{\beta}{2}-\omega\right)\hat x}.
	\end{split}\]
	Because $\kappa\geq\frac{\beta}{2}+\omega$, the above right-hand side
	is larger than~$\frac1\kappa$. The proof of the lemma is thereby achieved in this case.
	
	\medskip	
	{\em Case $\beta^2=4\gamma$.}
	
	\noindent 
	Observe preliminarily that $\beta,\gamma\neq0$, because otherwise $\kappa=0$.
	The solution $w$ is given by
	$$w(x)=xe^{-\frac\beta{2}x}.$$
	We see that $w'(\hat x)=0$, with $\hat x=\frac{2}\beta$.
	Direct computation reveals that
	$$w(\hat x)e^{\kappa\hat x}=\frac{2}\beta e^{-1+\frac{2}\beta \kappa},$$
	which is larger than $\kappa^{-1}$ because  $\kappa>\frac{\beta}{2}$.
	
	\medskip
	{\em Case $\beta^2<4\gamma$.}
	
	\noindent 
	The solution $w$ is now given by
	$$w(x)=\frac1\omega\, e^{-\frac{\beta}{2}x}\sin(\omega x),
	\quad\text{with }\ \omega:=\frac{\sqrt{4\gamma-\beta^2}}{2}.$$
	Then, $\hat x=\frac\pi{2\omega}$, and there holds 
	$$w(\hat x)e^{\kappa\hat x}=
	\frac1\omega\, e^{\left(\kappa-\frac{\beta}{2}\right)\hat x},$$
	which is larger than $\kappa^{-1}$ because $\kappa\geq\omega$ as well as 
	$\kappa>\frac\beta{2}$. 
\end{proof}	


\begin{proof}[Proof of \thm{1}]
	Take $x_0>0$. Suppose that $|u(x_0)|\neq 0$, otherwise the result trivially holds.
	Fix $c\in(0,|u(x_0)|)$ and define
	$$\bar x:=\sup\left\{x\geq x_0\ :\ |u(x)|>c e^{-\kappa(x-x_0)}\right\}.$$
	By continuity we know that $\bar x>x_0$.
	Assume by way of contradiction that $\bar x<+\infty$.
	This means that $|u(\bar x)|=c e^{-\kappa(\bar x-x_0)}$ and
	there exists a sequence $\seq{x}$ such that $x_n\nearrow\bar x$ as $n\to\infty$ and
	$|u(x_n)|>c e^{-\kappa(x_n-x_0)}$. 
	Up to replacing $u$ with $-u$ if need be, it is not restrictive to assume that 
	$u(\bar x)>0$, and thus $u>0$ in $[x_{\bar n},\bar x]$ for some $\bar n\in\N$.
	Since $u$ and $c e^{-\kappa(x-x_0)}$ are respectively a supersolution and a 
	solution of $\L_* =0$,
	it follows from the maximum principle of Proposition~\ref{pro:MP}
	that $u(x)\geq c e^{-\kappa(x-x_0)}$ for $x\in(x_{\bar n},\bar x)$. 
	Actually, the second statement of the proposition 
	implies that the inequality is strict, because it is strict at 
	$x_{\bar n}$, and in addition there holds that
	$$u'(\bar x)<-\kappa c e^{-\kappa(\bar x-x_0)}=-\kappa u(\bar x).$$
	We can finally apply the `bouncing' Lemma \ref{bouncing}, which provides us with some
	$h>0$ such that
	$$-u(\bar x+h)>u(\bar x)e^{-\kappa h}=c e^{-\kappa(\bar x-x_0+ h)}.$$
	This contradicts the definition of $\bar x$.
	
	We have thereby shown that $\bar x=+\infty$, that is,
	$$\varlimsup_{x\to+\infty}|u(x)|e^{\kappa x}\geq c e^{\kappa x_0}.$$ 
	This concludes the proof, due to the arbitrariness of $c\in(0,|u(x_0)|)$.
\end{proof}	
	We conclude the study of the $1$-dimensional case with an estimate of
	the distance between points where $u$ `does not decay too fast'.
\begin{proposition}\label{dense}
	Let $u$ be a solution of $\L u=0$ in $\R^+$ satisfying $u(0)=1$. 
	Then, for every
	$$
		\kappa'>\sup\frac{|q|}{2\alpha}+\sqrt{\sup\frac{|q|^2}{4\alpha^2}
		+\sup\frac{|V|}\alpha},
	$$
	there exists $h>0$ depending on $\|q/\alpha\|_\infty$, 
	$\|V/\alpha\|_\infty$ and $\kappa'$
	such that 
	$$\sup_{\bar x<x<\bar x+h}\big(|u(x)|-e^{-\kappa' x}\big)>0
	\quad\text{for all }\;\bar x>0.$$
\end{proposition}

\begin{proof}
	Assume by way of contradiction that there exist some functions $\seq{\alpha}$,
	$\seq{q}$, $\seq{V}$, $\seq{u}$ such that	
	$$\alpha_n (x)u_n''+q_n(x)u_n'+V_n(x)u_n=0\inn\R^+,$$  
	with
	$$\frac{|q_n|}{\alpha_n}\leq\beta,\quad
	\frac{|V_n|}{\alpha_n}\leq\gamma,$$	
	and moreover $u_n(0)=1$ and
	$$|u_n(x)|\leq e^{-\kappa' x}\for x\in[x_n,x_n+n],$$
	for some $x_n\geq0$.
	Up to replacing $u_n$ with $-u_n$ and decreasing $x_n$ if need be,
	we can assume without loss of generality that $u_n(x_n)=e^{-\kappa' x_n}$.
	Consider the functions $\seq{v}$ defined by 
	$$v_n(x):=u_n(x_n+x)e^{\kappa' x_n}.$$
	They satisfy some linear equations of the form
	$$v_n''+\t q_n(x)v_n'+\t V_n(x)v_n=0\inn\R^+,$$ 
	with $|\t q_n|\leq\beta$, $|\t V_n|\leq\gamma$, together with
	$v_n(0)=1$ and
	$$|v_n(x)|\leq e^{-\kappa' x}\for x\in[0,n].$$
	We now use standard elliptic estimates.
	They imply that the $\seq{v}$ are uniformly bounded in $W^{2,p}((0,R))$,
	for all $p<+\infty$ and $R>0$, and thus in 	$C^{1,\delta}([0,R])$, $\delta\in(0,1)$,
	by Morrey's inequality.
	We can then pass to the (weak) limit in 
	the inequalities $\L_* v_n\leq0\leq\L^* v_n$ and we find that (up to subsequences)
	$\seq{v}$ converges locally uniformly in $[0,+\infty)$ to a function 
	$v\in  W^{2,p}_{loc}(\R^+)\cap C^1([0,+\infty))$ satisfying $\L_* v\leq0\leq\L^* v$ and moreover
	$v(0)=1$ and $|v(x)|\leq e^{-\kappa' x}$ for $x>0$. We deduce 
	in particular that $v'(0)\leq-\kappa'<-\kappa$. It then follows from Lemma \ref{bouncing}
	that $v(h)<-e^{-\kappa h}$ for some~$h>0$, which is impossible because
	$|v(h)|\leq e^{-\kappa' h}$.
\end{proof}
	

\bigskip
\section{The radial cases}

We now turn to the $N$-dimensional case, considering first radial solutions.
\begin{proof}[Proof of Corollary \ref{cor:urad}]
	Assume by contradiction that there exists a nontrivial, radial solution $u(x)=\phi(|x|)$
	such that $\phi(r)\prec e^{-\kappa r}$ for some $\kappa$ satisfying
	$$\kappa>\varlimsup_{|x|\to+\infty}\frac{|q|}{2\alpha}+\sqrt{\varlimsup_{|x|\to+\infty}
		\frac{|q|^2}{4\alpha^2}
		+\varlimsup_{|x|\to+\infty}\frac{|V|}\alpha}.$$
	The function $\phi$ belongs to $W^{2,N}_{loc}((R_0,+\infty))$, 
	where $R_0>0$ is such that $\R^N\setminus B_{R_0}\subset\O$. 
	Let~$e_1$ be the first vector of the canonical basis of $\R^N$.
	For $r>R_0$, we compute 
	$$\L u(re_1)=A_{11}(re_1)\phi''(r)+\bigg(q_1(re_1)+\frac{\Tr A(re_1)-A_{11}(re_1)}{r}\bigg)\phi'(r)
	+V(re_1)\phi(r)=0.$$
	Namely, $\phi$ satisfies the equation $\t\L\phi=0$ in $(R_0,+\infty)$, where
	$$\t\L\phi:=\t \alpha(r)\phi''+\t q(r)\phi'+\t V(r)\phi,$$
	with 
	$$\t\alpha(r):=A_{11}(re_1),\qquad \t q(r):=q_1(re_1)+\frac{\Tr A(re_1)-A_{11}(re_1)}{r},
	\qquad\t V(r):=V(re_1).$$
	We have that $\t\alpha(r)\geq\alpha(re_1)$, where $\alpha(x)$ is the smallest eigenvalue
	of $A(x)$. Therefore,
	$$\frac{|\t q(r)|}{2\t\alpha(r)}\leq \sup_{\R^N\setminus B_r}\frac{|q|}{2\alpha}+Cr^{-1},
	\qquad \frac{|\t V(r)|}{\t \alpha(r)}\leq\sup_{\R^N\setminus B_r}\frac{|V|}{\alpha},$$
	where $C$ only depends on $N$ and the $L^\infty$ norm of the coefficients of $A$.
	In particular, for $R>R_0$ sufficiently large, there holds
	$$\kappa>\sup_{r>R}\frac{|\t q|}{2\t \alpha}+\sqrt{\sup_{r>R}
		\frac{|\t q|^2}{4\t\alpha^2}
		+\sup_{r>R}\frac{|\t V|}{\t\alpha}}.$$
	As a consequence, applying \thm{1} to the operator $\t\L$, we infer that
	$\phi=0$ in~$(R,+\infty)$. We would like to conclude from this that $u\equiv0$ in $\O$
	by means of the unique continuation property.
	However, the matrix $A$ being only in $L^\infty(\O)$, we
	are not in the regularity framework where such result applies.
	We overcome this difficulty by reducing to the $1$-dimensional case.
	Consider any $\hat R>0$ such that $\phi$ is defined in $(\hat R,+\infty)$.
	It satisfies there 
	$$|\phi''|\leq C'(1+\hat R^{-1})|\phi'|+C''|\phi|,$$
	for some $C',C''>0$. We can now apply the unique continuation property of 
	\cite{Continuation},~or even the classical Carath\'eodory theorem
	for ODEs, to deduce that $\phi\equiv0$ in $(\hat R,+\infty)$. 
	By the arbitrariness of $\hat R$, this means that
	$u\equiv0$ in $\O$, contradicting our initial assumption.
\end{proof}
	
	Next, we consider radial operators.
	In the sequel, $S=\partial B_1$ stands for the unit sphere in $\R^N$ centred at the origin
	and we let $dS$ denote its surface element.  

\begin{proof}[Proof of \thm{Lrad}]
	Assume by contradiction that there exists a nontrivial solution satisfying
	$u(x)\prec e^{-\kappa |x|}$ with
	$$\kappa>\varlimsup_{r\to+\infty}\frac{|q|}{2}+\sqrt{\varlimsup_{r\to+\infty}
		\frac{|q|^2}{4}
		+\varlimsup_{r\to+\infty}|V|}.$$
	It is convenient to rewrite the equation in spherical coordinates.
	Let $R_0$ be such that $\R^N\setminus B_{R_0}\subset\O$
	and let $\t u:(R_0,+\infty)\times S\to\R$ be the expression for $u$ in spherical coordinates,
	i.e., $\t u(r,\sigma):=u(r\sigma)$.
	The Laplace operator rewrites as follows:
	$$\Delta u(r\sigma)=\frac1{r^{N-1}}\partial_r(r^{N-1}\partial_r\t u)+\frac1{r^2}\Delta_\sigma\t u,$$
	with $\Delta_\sigma$ indicating the Laplace-Beltrami operator on the sphere $S$. 
	Then, using the identity $\partial_r\t u(r,\sigma)=\sigma\.\nabla u(r\sigma)$, we find that
	$$\L u(r\sigma)=
	\frac1{r^{N-1}}\partial_r(r^{N-1}\partial_r\t u)+\frac1{r^2}\Delta_\sigma\t u
	+q(r)\partial_r\t u(r,\sigma)+V(r)\t u=0,\quad \ r>R_0,\ \sigma\in S.$$
	The eigenvalues of $-\Delta_\sigma$ (counted with their multiplicity) are given by
	$$0=\lambda_1<\lambda_2\leq\lambda_3\leq\cdots$$
	Let $\vp_1\equiv1,\vp_2,\vp_3,\dots$ be the corresponding eigenfunctions,
	with $L^2(S)$ norm equal to $1$.
	We would like to multiply the equation for $\t u$ by $\vp_j$, $j=1,\dots$ , 
		and integrate it on the sphere, in order to get an 
	ODE for the projections $u_j$ defined~by
	$$u_j(r):=\int_{S}\t u(r,\sigma)\vp_j(\sigma)\, dS_\sigma.$$
	This cannot directly be done because $\partial_{rr}\t u$, $\Delta_\sigma\t u$
	are just in $L^p_{loc}((R_0,+\infty)\times S)$, for all $p<+\infty$.
	The lower order terms do not pose any problem because
	$\t u\in C^1((R_0,+\infty)\times S)$ by Morrey's inequality, and thus
	$$u_j'(r)=\int_{S}\partial_r\t u(r,\sigma)\vp_j(\sigma)\, dS_\sigma.$$
	In order to derive the equation for $u_j$,  we consider
  	$\psi\in C^\infty_c((R_0,+\infty))$ and compute
	\[\begin{split}
	\int_{\R^N\setminus\ol B_{R_0}}(\Delta u)\vp_j\bigg(\frac x{|x|}\bigg)\psi(|x|)\,dx
	& =\int_{\R^N\setminus\ol B_R} u\, \Delta\left(\vp_j\bigg(\frac x{|x|}\bigg)\psi(|x|)\right)dx\\
	&= \int_{R_0}^{+\infty}dr \int_{S}
	\t u \Big( \partial_r(r^{N-1}\psi'(r))-r^{N-3}\lambda_j\psi(r)\Big)\vp_j(\sigma) d S_\sigma\\
	&= \int_{R_0}^{+\infty}u_j(r)\Big( \partial_r(r^{N-1}\psi'(r))-r^{N-3}\lambda_j\psi(r)\Big)dr\\
	&= \int_{R_0}^{+\infty}\Big(-u_j'(r) r^{N-1}\psi'(r)-u_j(r)r^{N-3}\lambda_j\psi(r)\Big) dr.
		\end{split}\]
	On the other hand, using the equation $\L u=0$, we get
	\[\begin{split}
	\int_{\R^N\setminus\ol B_{R_0}}(\Delta u)\vp_j\bigg(\frac x{|x|}\bigg)\psi(|x|)\,dx &=
	-\int_{R_0}^{+\infty}dr\int_{S}\big(q(r)\partial_r\t u+
	V(r)\t u\big)\vp_j(\sigma)\psi(r)r^{N-1}\,dS_\sigma\\
	&=-\int_{R_0}^{+\infty}r^{N-1}\big(q(r)u_j'(r)+V(r)u_j(r)\big)\psi(r)\,  dr. 
	\end{split}\]
	This means that the following equalities 
	hold in the distributional sense in $(R_0,+\infty)$:
	$$-q(r)u_j'-\left(V(r)-\frac{\lambda_j}{r^2}\right)u_j=\frac1{r^{N-1}}(r^{N-1}u_j')'=u_j''+\frac{N-1}r u_j'.$$
	It follows in particular that $u_j''\in L^\infty((R_0,+\infty))$ and thus the equation is satisfied a.e.
	The coefficients of this equation fulfil 
	$$\frac12\sup_{r>R}\left|\frac{q(r)}2+\frac{N-1}{2r}\right|
	+\sqrt{\sup_{r>R}\left|\frac{q(r)}2+\frac{N-1}{2r}\right|^2
		+\sup_{r>R}\left|V(r)-\frac{\lambda_j}{r^2}\right|}<\kappa,$$
	provided $R$ is larger than some $R_j$.
	Therefore, because
	$$|u_j(r)|\leq \|\t u(r,\.)\|_{L^2(S)}\leq
	\sqrt{|S|}\,\|\t u(r,\.)\|_{L^\infty(S)}\prec e^{-\kappa r},$$
	the $1$-dimensional \UCI\ property (\thm{1}) entails that $u_j(r)=0$ for $r>R_j$. 
	Then, owing to the unique continuation property for ODEs, 
	we have that $u_j\equiv 0$ in $(R_0,+\infty)$.
	Finally, being $\seq{\vp}$ an Hilbert basis for $L^2(S)$, we know that 
	$\t u(r,\sigma)=\sum_{j=1}^{+\infty} u_j(r)\vp_j(\sigma)$ in the $L^2$ sense, and thus a.e.
	We have thereby shown that $\t u(r,\sigma)=0$ for $r>R_0$. Applying again the 
	unique continuation property, this time for equations in dimension~$N$ 
	with leading term given by the
	Laplace operator, see \cite{Continuation,Hormander}, 
	we eventually conclude that $u\equiv0$ in $\O$.	
\end{proof}

\begin{remark}
	Looking at the proof of \thm{Lrad}, one realizes that more general second order terms than
	the Laplace operator are allowed. Namely, those expressed in spherical coordinates by
	$$\frac1{r^{N-1}}\partial_r(r^{N-1}\partial_r\t u)+\frac{\vt(r)}{r^2}\Delta_\sigma\t u,$$
	for a given function $\vt$, not necessarily positive. 
\end{remark}




\bigskip
\section{Positive supersolutions}\label{sec:parabolic}

In this section, we consider a parabolic equation associated with the operator
$$\P u=\partial_t u-\Tr(A(x,t)D^2 u)+q(x,t)\.Du+V(x,t)u.$$
We always assume that $A,q,V$ are in $L^\infty$ 
and that $A$ is continuous and uniformly elliptic, i.e., that the function
$$\alpha(x,t):=\min_{\xi\in \R^N\setminus\{0\}}\frac{A(x,t)\xi\.\xi}{|\xi|^2}$$
is bounded from below away from $0$.
Solutions, subsolutions and supersolutions
are now assumed to be in $L^{N+1}_{loc}$ with respect to the $(x,t)$ variable, 
together with their derivatives $Du,D^2 u,\partial_t u$.

Here is our main result concerning positive supersolutions, from which the other results
of the section readily follow.
It is achieved using a refinement of the argument of the proof of \cite[Lemma 3.1]{RR}.

\begin{theorem}\label{thm:ancient}
Let $u$ satisfy
$$\P u\geq 0,
\quad x\in\O,\ t<0,$$
where $\O$ is an exterior domain in $\R^N$, together with
$$
\inf_{\su{x\in K}{t<0}}u(x,t)>0,
$$
for any compact set $K\subset\O$.
Then, 
	$$u(x,0)\succ e^{-\kappa|x|},$$
	for all $\kappa$ satisfying
	\Fi{k2p}
	\kappa>\varlimsup_{|x|\to\infty}\left(\sup_{t<0}
	\bigg(\frac{|q|}{2\alpha}+\sqrt{\frac{|q|^2}{4\alpha^2}
		+\frac{|V|}\alpha}\,\bigg)\right).
	\Ff
\end{theorem}

\begin{proof}
Let $R_0>0$ be sufficiently large so that $\R^N\setminus B_{R_0}\subset\O$. 
Take $\kappa$ satisfying \eqref{k2p}
and consider the function $\chi:\R^+\times\R^+\to\R$ defined by
\[ 
\chi(r,s):=
\begin{cases}
\displaystyle
\left(1-\frac{r}{s}\right)^{\kappa s} &
\text{if }0\leq r<s\\
0 & \text{if }r\geq s.
\end{cases}\]
Then define the function $\eta$ as follows:
\[\eta(x,t):=\chi(|x|-R\,, \delta t+\delta^{-1}+h),\]
where the parameters $R\geq R_0$ and $\delta,h>0$ will be chosen later.
The key properties of $\eta$ are that
it is compactly supported in space and that it converges to the function
$e^{-\kappa(|x|-R)}$ as the parameter $\delta$ tends to $0$.
We now proceed in two steps: first showing that $\eta$ is a
(generalised) subsolution of $\P=0$, next comparing it with~$u$.
\smallskip

\step{1} The function $\eta$ is a generalised subsolution of $\P=0$
provided $R,h$ are sufficiently large and $\delta$ is sufficiently small.
\medskip
\newline
For $-\delta^{-2}<t<0$ and $R<|x|<R+\delta t+\delta^{-1}+h$, we 
compute
$$\P\eta = \delta\,\partial_s\chi
-\frac{Ax\.x}{|x|^2}\,\partial_{rr}\chi
-\bigg(\frac{q\.x}{|x|}+\frac{\Tr A}{|x|}-\frac{Ax\.x}{|x|^3}\bigg)\partial_{r}\chi
	-V\chi.$$
Here and in what follows,
the functions $A,q,V$ and $\alpha$ are evaluated at $(x,t)$, 
whereas $\chi$ and its derivatives at $(r,s)=(|x|-R\,, \delta t+\delta^{-1}+h)$, which satisfy 
$$0<r<s\quad\text{and}\quad h<s<h+\delta^{-1}.$$
We impose $h\geq1/\kappa$, so that 
$\partial_{rr}\chi\geq0$. We then obtain
$$\P\eta_n \leq \delta\,\partial_s\chi-\alpha\,\partial_{rr}\chi
-\bigg(|q|+\frac{C}{|x|}\bigg)\partial_{r}\chi+|V|\chi,$$
with $C$ only depending on $N$ and the $L^\infty$ norm of the coefficients of $A$.
Observing~that 
$$\partial_r\chi=-\kappa\frac{s}{s-r}\,\chi,\qquad
\partial_{rr}\chi=\kappa\left(\kappa-\frac1s\right)\left(\frac{s}{s-r}\right)^2\chi,$$
$$\partial_s\chi=\kappa\chi\left(\log\left(1-\frac r s\right)+\frac{r}{s-r}\right)\leq
\kappa\chi\frac{s}{s-r},$$
we eventually derive
\[\begin{split}
\frac{\P\eta}{\chi} &\leq
\delta\kappa \frac{s}{s-r}-\alpha\kappa\left(\kappa-\frac1s\right) \left(\frac{s}{s-r}\right)^2
+\kappa\bigg(|q|+\frac{C}{|x|}\bigg)\frac{s}{s-r}+|V|\\
&= \left(\frac{s}{s-r}\right)^2\left(-\alpha\kappa^2
+\kappa\bigg(|q|+\frac{C}{|x|}+\delta\bigg)\frac{s-r}{s}+\frac\alpha s\,\kappa+
|V|\bigg(\frac{s-r}{s}\bigg)^2\right)\\
&\leq \left(\frac{s}{s-r}\right)^2\left(-\alpha\kappa^2
+\kappa\bigg(|q|+\frac{C}{R}+\delta+\frac\alpha{h}\bigg)+|V|\right)
\end{split}\]
We need to show that the right-hand side is less than or equal to $0$.
For this, we use~\eqref{k2p} which allows us to rewrite $\kappa=\t\kappa+\e$, with $\e>0$
and
$$\tilde \kappa:=\sup_{{|x|>R_1}\atop{t<0}}
\bigg(\frac{|q|}{2\alpha}+\sqrt{\frac{|q|^2}{4\alpha^2}
		+\frac{|V|}\alpha}\,\bigg),$$
for some sufficiently large $R_1>R_0$.
The quantity $\t\kappa$ is the supremum with
respect to $|x|>R_1$, $t<0$ of the largest root of the polynomial 
$$Q_{x,t}(X):=\alpha(x,t)X^2-|q(x,t)|X-|V(x,t)|.$$
It follows that $Q_{x,t}(\t\kappa)\geq0$ if $|x|>R_1$ and $t<0$, whence
\[\begin{split}Q_{x,t}(\kappa)
&= Q_{x,t}(\t\kappa)+\alpha(x,t)(2\t\kappa\e+\e^2)-|q(x,t)|\e\\
&\geq (2\alpha(x,t)\t\kappa-|q(x,t)|)\e+\alpha(x,t)\e^2\\
&\geq\alpha(x,t)\e^2,
\end{split}\]
where the last inequality follows from the explicit expression for $\t\kappa$.
As a consequence, for $|x|>R>R_1$ and $t<0$, we have that
$$-\alpha\kappa^2
+\kappa\bigg(|q|+\frac{C}{R}+\delta+\frac\alpha{h}\bigg)+|V|
\leq-\big(\inf\alpha)\e^2+\frac{\kappa C}{R}+\kappa\delta+\frac{\kappa\alpha}{h},$$
which is a negative constant provided
$R,h$ are large enough and $\delta$ is small enough.
In the end, under such conditions, there holds
$$\P\eta\leq0\quad\text{for }\;-\delta^{-2}<t<0,\ R<|x|<R+\delta t+h.$$

\step{2} Comparison between $u$ and $\eta$.
\medskip
\newline
By the previous step, we can take $R,h>0$ such that $\eta$ is 
a generalised subsolution of $\P=0$ for $-\delta^{-2}<t<0$ and $|x|>R$,
provided $\delta$ is sufficiently small.
By hypothesis, we can renormalise $u$ in such a way that 
$$\inf_{\su{R\leq |x|\leq R+h}{t<0,\,}}u=1.$$
At the time $t=-\delta^{-2}$, we have that
$$\eta(x,\delta^{-2})=\chi(|x|-R\,, h),$$
which is bounded from above by $1$
and vanishes for $|x|\geq R+h$.
Hence, $u(x,\delta^{-2})\geq \eta(x,\delta^{-2})$ for $|x|\geq R$.
We further have that
$u(x,t)\geq \eta(x,t)$ for $|x|=R$, $t<0$. 
Thus, for $\delta$ small enough, 
applying the parabolic comparison principle in the set
$(\R^N\setminus B_R)\times(-\delta^{-2},0)$,
we infer that $u\geq\eta$  in this set, and then in particular
\[u(x,0) \geq\eta(x,0)=\chi(|x|-R\,,\delta^{-1}+h).\]
Recalling the expression of $\chi$, for $R\leq |x|\leq R+\delta^{-1}+h$
we compute the above right-hand side getting
$$\chi(|x|-R\,,\delta^{-1}+h)=
\left(1-\frac{|x|-R}{\delta^{-1}+h}\right)^{\kappa (\delta^{-1}+h)},$$
which tends to $e^{-\kappa(|x|-R)}$ as $\delta\to0^+$.
This shows that $u(x,0)\geq e^{-\kappa(|x|-R)}$ for $|x|\geq R$, with $R$ depending on $\kappa$.
Since this is true for every $\kappa$ satisfying \eqref{k2p}, the proof is complete.
\end{proof}

\begin{proof}[Proof of \thm{u>0}] 
	One just applies \thm{ancient}
	to the stationary supersolution~$u$ of the equation $\partial_t u-\L u=0$.
	Observe that $\inf_{K}u(x)>0$ for any compact set $K\subset\O$, because
	$u$ is positive and continuous.
\end{proof}	

\begin{proof}[Proof of \thm{l1}]
	We know from \cite[Theorem 1.4]{BR4} that the 
	generalised \pe\ $\l$ admits a positive eigenfunction $\vp$, that is, 
	a positive solution of $-\L\vp=\l\vp$ in $\O$. Moreover, 
	$\vp=0$ on $\partial\O$ if $\O$ is smooth ($\neq\R^N$).
	Because $\l\geq0$, we have
	$$-\L\vp=\l\vp\geq0\inn\O.$$
	\thm{u>0} then implies that $\vp(x)\succ e^{-\kappa|x|}$,
	for all $\kappa$ satisfying \eqref{k2}.
	Assume by contradiction that there exist a nontrivial solution $u$ of
	$\L u=0$ in $\O$ and $\kappa$ satisfying~\eqref{k2} such that
	$$\varlimsup_{|x|\to+\infty}|u(x)|e^{\kappa|x|}<+\infty.$$
	Then, for $\kappa'<\kappa$ still satisfying \eqref{k2}, we 
	have that $u(x)\prec e^{-\kappa'|x|}\prec\vp(x)$. We
	claim that this entails $u\geq0$, whence $u>0$ by the \SMP,
	which is a contradiction due to~\thm{u>0}.
	
	We prove this claim by distinguishing the two different hypotheses
	of the theorem.
	
	\medskip
	{\em Case $\O=\R^N$}.\\
	Suppose that $u<0$ somewhere. Then, since $u(x)\prec\vp(x)$, the quantity
	$$C:=\max_{\R^N}\frac{-u}{\vp}$$
	is a well defined positive number. It follows that 
$\min_{\R^N}(C\vp+u)=0$.
	The strong maximum principle then yields $C\vp\equiv - u$, which  
	is impossible because $u\prec\vp$.
	This shows that necessarily $u\geq0$.
	
	\medskip
	{\em Case $\O\neq\R^N$ and \;$\varliminf_{\,x\to\partial\O}u(x)\geq0$}.\\ 
	The key tool here is the maximum principle in small domains given by \cite[Theorem~2.6]{BNV}.
	It provides some $\delta>0$ such that the maximum principle holds in any domain 
	with measure smaller than $\delta$.
	
	Consider the sequence of open sets $\seq{\O}$ defined by
	$$\O_n:=\{x\in\O\ :\ \dist(x,\partial\O)>1/n\}.$$
	Since $\O$ is an exterior domain, we have that $|\O\setminus\O_n|\to0$ as $n\to\infty$,
	whence $|\O\setminus\O_{\bar n}|<\delta$ for some $\bar n\in\N$.
	Next, using the fact that $u\prec\vp$ and that $\vp>0$ in $\O\supset\ol\O_{\bar n}$, 
	we can find $C>0$ such that $C\vp+u\geq0$ in $\ol\O_{\bar n}$.
	Finally, in the open set $\mc{O}:=\O\setminus\ol\O_{\bar n}$, the function $w:=C\vp+u$
	is a supersolution of $\L=0$ satisfying
	$$\varliminf_{x\to\partial\mc{O}}w(x)\geq0.$$
	Applying the maximum principle \cite[Theorem 2.6]{BNV} in every connected component of~$\mc{O}$,
	we find that $C\vp+u\geq0$ in $\mc{O}$ too.
	We have shown that $C\vp+u\geq0$ in $\O$. 
	
	Define
	$$C^*:=\inf\{C>0\ :\ C\vp+u\geq0 \text{ in }\O\}.$$
	Assume by contradiction that $C^*>0$. On one hand, there holds $C^*\vp+u\geq0$ in~$\O$.
	On the other, for any $\e\in(0,C^*)$, we necessarily have that $\inf_{\O_{\bar n}}((C^*-\e)\vp+u)<0$,
	because otherwise applying \cite[Theorem 2.6]{BNV} as before 
	we would get $(C^*-\e)\vp+u\geq0$ in $\O$, contradicting the definition of $C^*$. 
	There exists then a family of points
	$(x_\e)_{\e}$ in~$\O_{\bar n}$ for which $(C^*-\e)\vp(x_\e)+u(x_\e)<0$.
	This family is bounded as $\e\to0$ because $u\prec\vp$ and therefore it converges  (up to subsequences)
	to some $x_0\in\ol\O_n$. We deduce that $C^*\vp(x_0)+ u(x_0)\leq0$, and actually 
	$C^*\vp(x_0)+u(x_0)=0$ because $C^*\vp+u\geq 0$ in $\O$. The elliptic \SMP\ eventually yields
	$C^*\vp\equiv- u$, which is impossible because $u\prec\vp$. This shows that $C^*=0$. 
	Namely, $u\geq0$ in $\O$.
\end{proof}



\end{document}